\documentclass[12pt]{amsart}
\usepackage{latexsym,fancyhdr,amssymb,color,amsmath,amsthm,graphicx,listings,comment}
\usepackage[section]{placeins} 
\newtheorem{thm}{Theorem}  \newtheorem{propo}{Proposition} 
\newtheorem{question}{Question} 
\definecolor{red1}{rgb}{1,0.9,0.9} \definecolor{blue1}{rgb}{0.9,0.9,1} \definecolor{green1}{rgb}{0.9,1,0.9}
\def\question#1{ \vspace{2mm} \begin{center} \fcolorbox{blue1}{blue1}{ \parbox{13.2cm}{{\bf Question:} #1}} \vspace{2mm} \end{center} }

\let\paragraph\subsection
\def\B#1#2{{#1\choose #2}}

\title{Constant index expectation curvature for graphs or Riemannian manifolds}
\author{Oliver Knill}
\date{December 23, 2019}
\address{Department of Mathematics \\ Harvard University \\ Cambridge, MA, 02138 }

\begin{document}
\maketitle

\begin{abstract}
An integral geometric curvature $K_\mu$ is defined as the
index expectation $K(x)=E_{\mu}[i(x)]$ if a probability measure
$\mu$ is given on vector fields on a Riemannian manifold or on a finite simple graph.
We give examples of finite simple graphs which 
do not allow for any constant $\mu$-curvature and prove that for
one-dimensional connected graphs, there is a convex set of constant 
curvature configurations with dimension of the first Betti number of the graph. 
In particular, there is always a unique constant curvature solution for trees. 
\end{abstract}

\section{In a nutshell}

\paragraph{}
If a probability distribution $p_x$ is given on the vertex set $V_x$ of 
every complete sub-graph $x \in G$ of a finite simple graph $(V,E)$, 
one obtains a curvature $K(v) = \sum_{x \in G} p_x(v) \omega(x)$ with 
$\omega(x)=(-1)^{{\rm dim}(x)}$. 
Such a curvature $K$ satisfies the Gauss-Bonnet formula 
$\chi(G)=\sum_{x \in G} \omega(x) = \sum_{v \in V} K(v)$ for the Euler characteristic 
$\chi(G)$ of the graph $(V,E)$ or its simplicial complex $G$. 
If $p_x(v) \in \{ 0,1 \}$, this is a Poincar\'e-Hopf formula with integer 
index values $i(v)=K(v)$ \cite{poincarehopf,parametrizedpoincarehopf,PoincareHopfVectorFields}. 
If $p_x$ is the uniform distribution on each set $V_x$, it produces the 
Gauss-Bonnet-Chern integrand $K(v)=\sum_{x, v \in x} \omega(x)/(|x|+1)$ \cite{cherngaussbonnet}.
Let $\mathcal{C}$ denote the set of graphs for which a probability measure
$p_x$ on $V_x$ exists for each $x \in G$ such that $K(v)$ is constant. 
One can also ask to minimize the variance 
${\rm Var}[K]=\sum_{v \in V} (K(v)-m)^2/|V|$, where $m$ is $\chi(G)/|V|$ is the average
curvature. The question whether a given complex $G$ is in $\mathcal{C}$ is a linear
programming problem as it attempts to find solutions of a linear system of equations
under finitely many inequality conditions which assure that the $p_x(v)$ are in the interval
$[0,1]$ and that they add up to $1$.

\begin{figure}[!htpb]
\scalebox{0.4}{\includegraphics{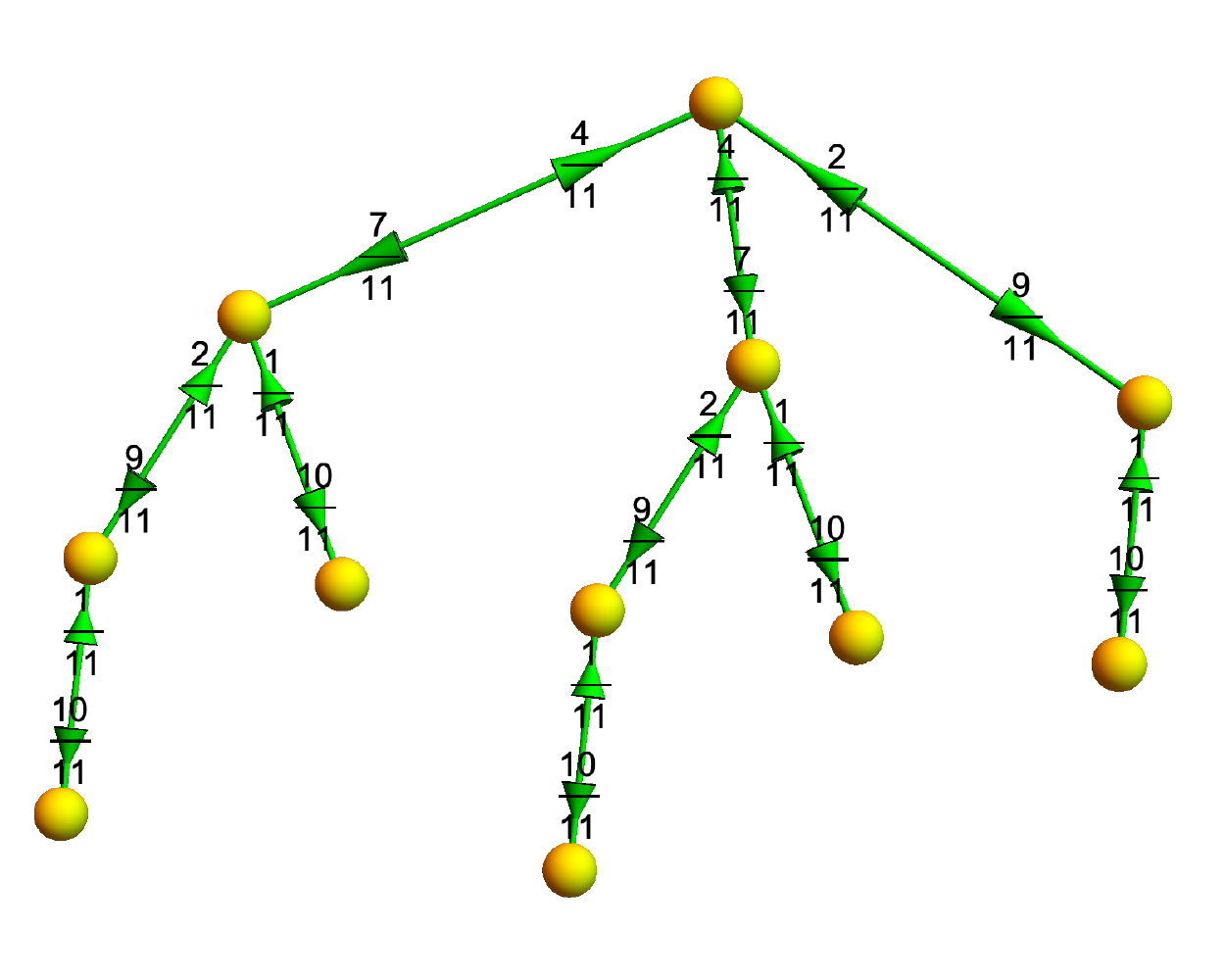}}
\scalebox{0.4}{\includegraphics{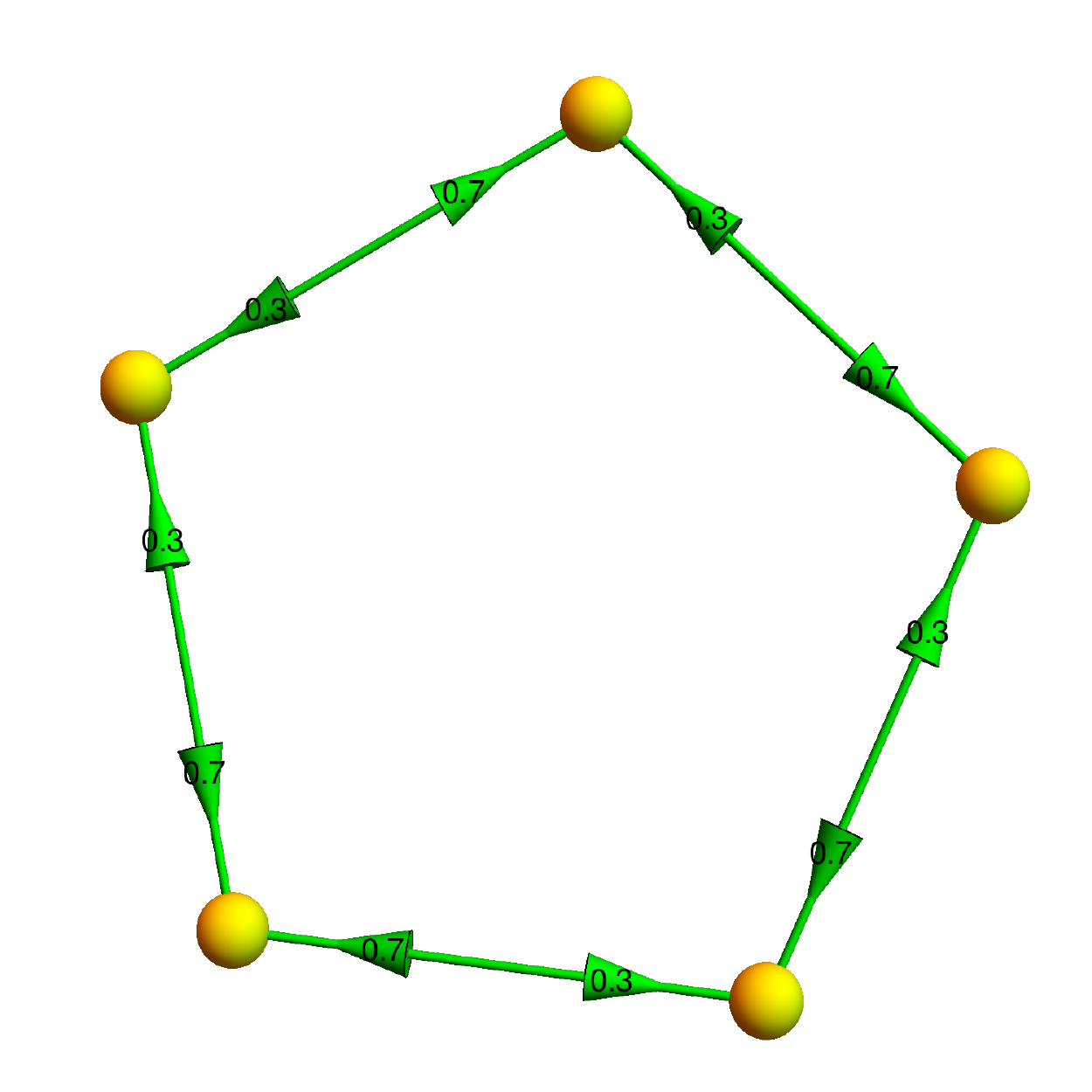}}
\scalebox{0.4}{\includegraphics{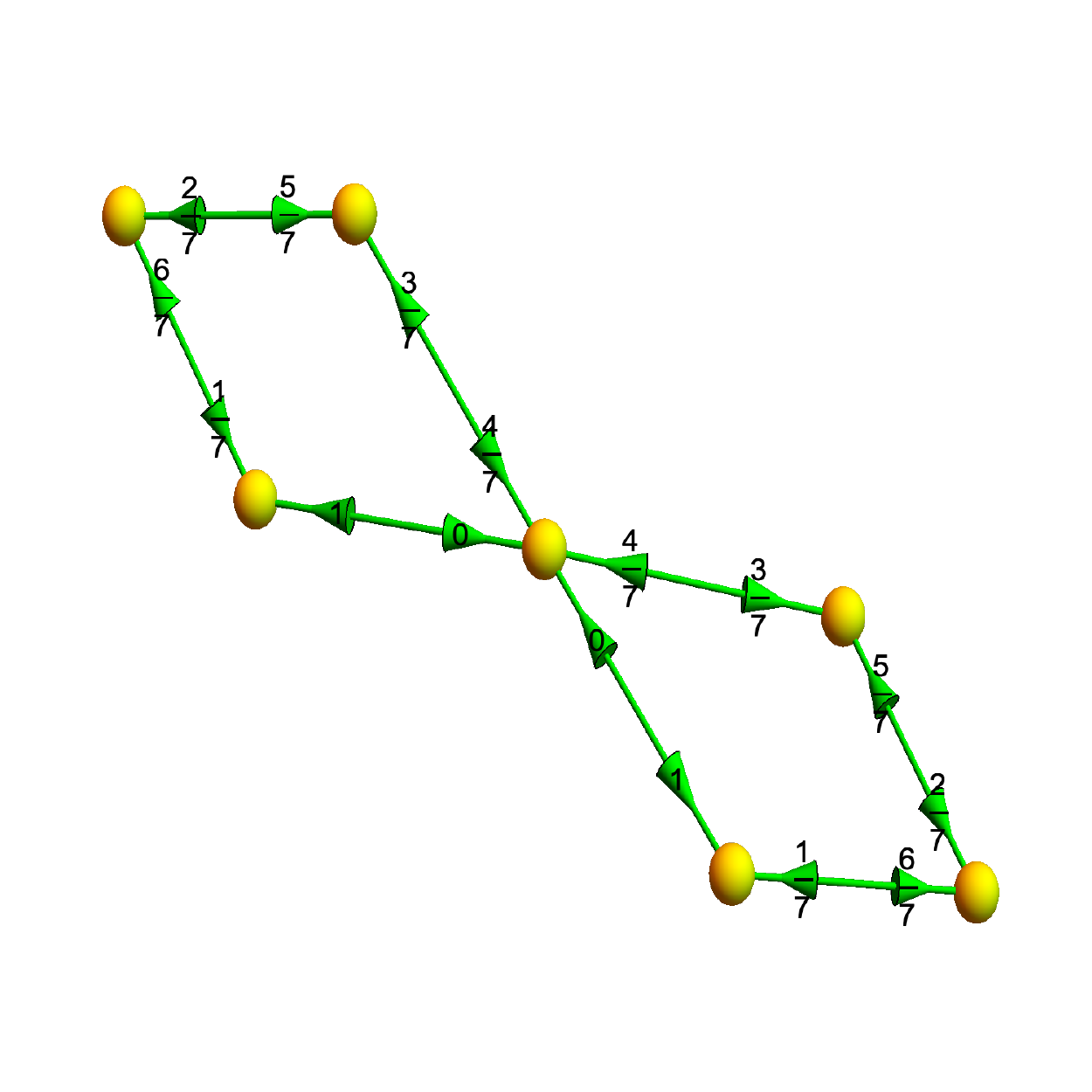}}
\label{treecircle}
\caption{
For a tree $G=(V,E)$, one can distribute the ``energy" value $\omega(x)=-1$ of every edge
$x \in E$ in a unique way to the two vertices. There is then a unique curvature 
$K(v)=1-\sum_{x \in E, v \in x} p_x(v)$ satisfying 
Gauss-Bonnet $\sum_{v \in V} K(v)=\chi(G)=1$ such that $K(v)$ is 
constant $1/|V|$. For a circular graph, there is an entire interval of solutions, just
take the probability space $p_x$ the same on each edge. We chose here $p_x=\{0.3,0.7\}$.
In the case of a figure-$8$ graph, we can even tune two parameters. 
}
\end{figure}

\paragraph{}
We prove here that triangle-free graphs are in $\mathcal{C}$ and give 
examples of graphs outside $\mathcal{C}$. The examples outside $\mathcal{C}$ 
are non-manifold like so far as we have no examples yet of $d$-graphs outside $\mathcal{C}$. 
We furthermore note that in the $1$-dimensional connected case, the solution set of 
probability distributions solving the constant curvature equation is a 
convex set of dimension $b_1=1-\chi(G)$. The integer $b_1$ is the only relevant Betti number for 
one-dimensional connected complexes. In particular, the solution set 
is unique for trees. One can then get probability distributions $p_x$ by 
considering a probability space $(\Omega,\mu)$ of locally 
injective functions $g$ (colorings) on the graph $G$ and get the curvature 
$K(v)=E_{\mu}[i]$ as index expectation. We have explored in
\cite{indexexpectation,colorcurvature} how to get the standard
Gauss-Bonnet curvature as index expectation. 

\paragraph{}
The constant curvature question can be ported to smooth compact manifolds by taking 
a probability space $(\Omega,\mu)$ of Morse functions on $M$ and defining curvature 
$K(v)=K_{\mu}(v)$ as index expectation. We have experimented with 
that (see e.g. \cite{eveneuler,KnillBaltimore}) as there are various natural measures
which can be defined as such like taking heat kernel functions $f(y) = e^{-t L}(x,y)$
and let $\Omega$ be the manifold itself with probability measure $dV$.
An important example of a curvature is the Gauss-Bonnet-Chern integrand for a compact 
Riemannian manifold $M$. It is not possible to realize a metric on $M$ in general which has constant
Gauss-Bonnet-Chern curvature, we have not yet found a compact connected manifold that 
can not be equipped with constant index expectation curvature. See Question~(\ref{constantcurvature}).

\paragraph{}
Back to the discrete case, one can look at the question for $d$-graphs.
These finite simple graphs which are discrete manifolds in the sense that 
they have the property that every unit sphere $S(v)$ is a $(d-1)$-sphere.
A $d$-sphere is then just a $d$-graph which when punctured becomes contractible. 
The constant curvature problem for discrete manifolds is not yet studied.
It relates to the question on how fast the Euler characteristic of a $d$-graph
with $n$ vertices can grow as a function of $n$. This is of independent interest.

\paragraph{}
We know that for general Erdoes-R\'enyi graphs in $E(n,p)$, 
the maximal Euler characteristic
in $E(n,p)$ grows exponentially along sub-sequences in $(p_k,n_k)$.
The reason is because the expectation value of the Euler characteristic on $E(n,p)$ is
explicitly given as
$$ {\rm E}_{n,p}[\chi] = \sum_{k=1}^n (-1)^{k+1} \B{n}{k} p^{\B{k}{2}} \;  $$
\cite{randomgraph}. 

\paragraph{}
But the growth for $d$-graphs (discrete $d$-manifolds) appears to be unknown: a possible 
super linear grow along a sub-sequence of $n$ would produce $d$-graphs for which the curvature 
can not be made constant. Formulated differently, if every discrete
$d$-manifold allowed for constant curvature then the maximal Euler characteristic
$\chi(G)$ were only grow linearly. Let $X_d(n)$ denote the maximal $\chi(G)$ which 
a $d$-graph with $n$ elements can have. Interesting is the following question: 

\question{ How fast does $X_d(n)$ grow for $n \to \infty$?  }

\section{A trade allegory} 

\paragraph{}
Before we start with the actual paper, let us look at the following distribution
problem for a finite network with nodes $V$ and connections $E$. It is equivalent
to the constant curvature problem we study here for one-dimensional networks. 

\paragraph{}
Consider the following cost distribution problem for a finite simple graph $(V,E)$: \\

{\bf Assume that each connection $(a,b) \in E$ between two nodes $a,b \in V$ 
costs a fixed amount $1$. How do we distribute the cost for each transaction $(a,b)$ 
to the two parties $a$ and $b$ in order that the total cost of all parties is the same? }

\paragraph{}
Our result shows that one can solve the fair distribution problem in a unique
way if the network is a tree and that there is a $b$-dimensional set of 
distribution parameters if the network has $b=b_1$ independent loops. The integer 
$b_1$ is the first Betti number of the graph. 

\paragraph{}
Now look at the case when the network also can have triangles (serving as two dimensional faces) but no
complete sub-graphs $K_4$ with $4$ vertices (these $K_4$ subgraphs are three dimensional tetrahedral 
simplices). Let $F$ be the set of triangles. The simplicial complex $G$ defined by the network now is
the union $V \cup E \cup F$ of zero, one and two dimensional parts.
The Euler characteristic is given by the Descartes formula
$\chi(G) = \sum_{x \in G} \omega(x) = |V|-|E| + |F|$. In the 
case of planar graphs where one can clearly extend the notion of face to other polygonal shapes, one has 
$\chi(G)=2$ which was first secretly recorded
by Ren\'e Descartes \cite{Aczel} and then proven by Euler for planar graphs \cite{Richeson}. 
A period of confusion \cite{lakatos} followed which can be attributed to 
definitions of polyhedra, especially also in higher dimensions \cite{Gruenbaum2003}. 

\paragraph{}
So, lets look at a finite simple graph in which there are no complete subgraphs $K_4$.
We think of it again as a trading network, in which exchanging stuff over some some connection
produces a fixed amount $1$ of cost. But now, each triangular clique produces a synergy as it allows to save
cost. Each trade triangle produces the same positive amount $1$ of profit which needs to be distributed to the
three players. The constant curvature problem is now equivalent to the following problem. 

\paragraph{} 
Cost distribution problem with synergy. \\

{\bf Assume each connection $(a,b)$ between two nodes produces a fixed amount $1$ of cost
and each triangular trade-relation generates the same fixed amount $1$ of synergy. How do 
we split the transaction cost for each transaction to the two players and how do we split 
the ``synergy bonus" for each triple of cooperating players so that every player has the same total? }

\paragraph{}
Now the situation is different and we can in general no more distribute things equally. Obviously, in 
part of the ``world", where better connections and more triangles are present, one can work
more effectively. If there are other parts, where the beneficial element of the triangular
synergy is missing, it is impossible to make up for the missing synergy in that part of the 
world. 

\paragraph{}
The fish graph displayed in Figure~(\ref{fishgraph}) illustrate this situation. 
In the main body of the fish there are triangles present which produces an obvious advantage there. 
In the swim fins, where no triangles are present, it is impossible to make up for the missing
benefit. The lack of connectivity is a handicap which can not be fixed by locally distributing
the costs more effectively. One would need a non-local redistributions (a development help so to
speak) in order to equalize the cost. Now, since curvature should always be a local quantity, 
this is not possible here. 

\paragraph{}
This simple trade model for cost distribution could be made more realistic or adapted to other 
networks. One way is to replace the ``topological cost" $\omega(x)= (-1)^{{\rm dim}(x)}$ leading to 
Euler characteristic with some arbitrary cost $H(x)$, a quantity we interpreted as ``energy" in 
\cite{EnergizedSimplicialComplexes}. The cost distribution problem now also depends on 
the value $H(v)$ of the nodes (which in the allegory is a measure for the wealth
of the player). 

\paragraph{}
In any case, the distribution problem is a linear programming problem. If there is a
solution, it can be found by a simplex method. As mentioned in the last section, one 
can in the case where no solution exists to minimize the variance. This is now a 
variational problem which again has constraints given by various inequalities. 

\section{Introduction}

\paragraph{}
Of classical interest in Riemannian geometry are {\bf spaces of constant curvature}. 
Especially well studied is the case of constant sectional curvature which leads to
space forms \cite{WolfConstantCurvature}. One can also study constant curvature curves, and other 
constant curvature manifolds, where the question of course depends on what
``curvature" is. As for curvature on manifolds, besides looking at sectional curvature
leading to ``constant curvature manifolds" one can also look at manifolds with 
constant Euler curvature (the curvature entering the Gauss-Bonnet-Chern theorem), 
constant Ricci curvature (leading to Einstein manifolds), constant scalar curvature 
or then constant mean curvature (which leads to minimal surfaces). Motivated from
physics, where curvature is associated with some sort of energy or mass, the concept
of constant curvature is some sort of equilibrium situation. 

\paragraph{}
A different kind of curvature is obtained by averaging Poincar\'e-Hopf
indices $i_g$ over a probability space $(\Omega,\mu)$ of 
Morse functions $g$ on $M$. The indices of Morse function are $\{-1,0,1\}$-valued 
divisors on $M$. They can also be seen as signed Dirac measures meaning pure point measures supported
on finite sets. The usual Gauss curvature of a Riemannian $2$-manifold $M$ is an 
example: Nash embed $M$ into a higher dimensional Euclidean space $E$ and take the probability 
space of all linear functions in $E$ which is rotational invariant. One of
historically earliest cases of curvature, the (solid) angle excess for convex polytopes 
geometrically realized in $E$ can be seen that way. 
For almost all linear functions on $E$ one has an index defined on the vertices of the polytop.
Averaging over all linear functions gives the solid angle.

\paragraph{}
Which manifolds allow for constant index expectation curvature? 
By the uniformization theorem, a $2$-manifold always allow for a constant Gauss
curvature in that way. The question whether there are
even-dimensional manifolds which do not allow for a Riemannian metric with constant 
Gauss-Bonnet-Chern curvature appears to be not studied so far. Maybe it is too obvious
that this is in general not possible:  we note that this can happen already for $4$-manifolds. 
But this is only given by example. We do not know for example for concrete cases like 
$M=S^2 \times S^2$ whether there exists a Riemannian metric $g$ on $M$ leading to constant 
Gauss-Bonnet-Chern curvature. 

\paragraph{}
In comparison, as the Hopf conjectures show, it is unknown whether there is a 
positive curvature metric on a space like $M=S^2 \times S^2$. The focus on 
the Gauss-Bonnet-Chern integrand has been abandoned maybe because in dimension $d \geq 6$
the algebraic Hopf conjecture (the question whether one achieve 
positive Gauss-Bonnet-Chern integrand if the manifold has positive curvature) 
has failed: there are positive curvature 6-manifolds for which the Gauss-Bonnet-Chern 
integrand can become negative at some places. (See e.g. \cite{Weinstein71}). 

\paragraph{}
We were led to the index expectation curvature also through
such algebraic questions (for us mostly in the discrete). 
Having many curvatures available renders the algebraic Hopf conjecture again interesting: 

\question{Does every even-dimensional compact connected positive curvature manifold allow
for some positive $\mu$-curvature?}

\paragraph{}
Of course the question must be difficult as answering it affirmatively would imply the Hopf conjecture
which history has shown to be difficult. 
The above question is intriguing because if a manifold has positive curvature then there is
a positive $\mu$ curvature on it (namely the constant curvature we establish).
An affirmative answer would show that 
the just formulated question does not only imply
but is equivalent to the Hopf conjecture: if the Hopf conjecture holds, then we
can realize a positive $\mu$-curvature.  

\paragraph{}
As for constant sectional curvature, we do not explore the discrete case except mentioning 
one simple case of $d$-graphs (discrete d-dimensional manifolds in which every unit sphere 
is a $(d-1)$-sphere) for which all embedded wheel graphs are isomorphic 
(have the same number $k$ of vertices). If one defines 
curvature as $1-k/6$ for such a section, we have looked at
positive curvature case in \cite{SimpleSphereTheorem} and shown that every 
positive curvature $d$-graph is necessarily a $d$-sphere. 

\paragraph{}
Which finite simple $d$-graphs have constant curvature {\bf in the strong 
sense} that all embedded wheel graphs are 
isomorphic? In the case $d=2$ only the octahedron (constant curvature $K=1/4$ on the 8 vertices) 
and the icosahedron (constant curvature $K=1/6$ on the 12 vertices).
In the case $d=3$, it follows from the classification of regular polytopes that 
there is only the 16-cell $4 S_0 = S_0+S_0+S_0+S_0$, (where $+$ is the join), 
and the 600-cell which have constant curvature. In dimensions $d>3$, there is only the $d$-dimensional
cross polytop $(d+1) S_0$ again using the Schl\"afli classification of regular platonic solids in $d$-dimensions.
The reason is that if $G$ has constant curvature, then every unit sphere must have constant curvature
which forces the graphs to be Platonic. In other words:

\begin{propo}
For $d=2$ and $d=3$, there are exactly two constant curvature $d$-graphs in the
strong sense. For $d \geq 4$, there is a unique constant curvature graph in the strong sense:
the $d$-dimensional cross polytop.
\end{propo}

\paragraph{}
We will weaken the constant curvature condition (the notion that all sectional curvatures
are positive) elsewhere to become more realistic. It uses of course 
 index expectation. It will allow to get discrete models of constant curvature which
look like constant curvature manifolds in the continuum and which also should lead to 
more realistic sphere theorems in the discrete. Of course, such theorems then would need
pinching conditions analogue to the continuum. The discrete case is then a play ground for
analogue Hopf questions. 

\paragraph{}
Integral geometric questions \cite{Santalo,Santalo1} have been studied also in combinatorial
settings \cite{KlainRota}. It produces an alternative to tensor calculus. 
It allows to define classical distances for example: if a probability measure $\mu$ is given
on the space of linear functions on an Euclidean space $E$ in which a Riemannian manifold $M$
is embedded, then the distance of a curve $\gamma$ can be measured as the expectation number
of the number of intersections of hyperplanes $f=0$ with $\gamma$. This Crofton approach 
\index{Crofton} recovers the Riemannian metric. It is more than natural also see curvature as an 
expectation, an expectation of Poincar\'e-Hopf indices. And this has also been done classically
for a while now \cite{Blaschke,Nicoalescu}.

\paragraph{}
For manifolds, taking probability measures on Morse functions is more convenient than taking
the measure on the larger space of vector fields with finitely many isolated non-degenerate equilibrium points.
On the other hand, taking a probability space on functions is less convenient in the discrete
and it is better to work with probability spaces of vector fields. The later leads to 
the frame work stated initially, where probability measures $p_x$ on simplices $x$ are given. 
The simplest set-up is to distributing the curvature values $\omega(x)$ from the simplices to the vertices.
This can be done if each simplex $x$ is equipped with a probability space $p_x$ 
\cite{PoincareHopfVectorFields}.

\section{Constant curvature manifolds}

\paragraph{}
Let us look now at the case of smooth manifolds and ask: 

\question{
Does every compact connected smooth manifold $M$ admit a constant index expectation curvature?
Is there is a probability space $(\Omega,\mu)$ of Morse functions $g: M \to \mathbb{R}$ 
such that the expectation $K(v)=E_{\mu}[i_g(v)]$ of the Poincar\'e-Hopf index $i_g$ divisors 
on $M$ is a constant function on $M$.
\label{constantcurvature}
}

\paragraph{}
The question is not interesting in dimension $d=1$ because $M$ is then a circle 
which has constant index expectation $0$, obtained by embedding $M$ as the standard 
circle in the plane and taking the probability space of all 
linear functions $f(x,y) = \cos(\theta) x + \sin(\theta) y$ of functions equipped with 
the uniform measure. This space induces a space of Morse functions on $M$. 

\paragraph{}
Now lets look at a general smooth, compact and connected manifold of dimension $d \geq 2$. 
Pick a Riemannian metric on $M$ (it is well known using a partition of unity that a 
smooth manifold can be equipped with a Riemannian metric by patching together metrics given
on each chart). This defines a volume measure on $M$. Normalize it so that it becomes 
a probability measure $dV$ on $M$. A theorem of Brin-Feldman-Katok \cite{BrinFeldmanKatok} assures 
that $M$ admits a smooth Bernoulli diffeomorphism $T$ with respect to such a volume measure. 
The automorphism $T$ on the probability space $(M,\mathcal{A},dV)$ is measure theoretically conjugated
to a Bernoulli system. This theorem needs that the dimension of $M$ is bigger or equal than $2$. 
Pick an arbitrary Morse function $g$ on $M$. Assume the indices of $g$ 
are supported on the set $m_1, \dots, m_n$ of critical points of $g$ in $M$
which are generic with respect to $T$ in the sense that all the orbits $T^k(m_j)$ are
uniformly distributed on $M$ (using a partition of unity it is easy to possibly modify the 
critical points $m_j$ if they would not be generic. The points which are generic in the sense 
of ergodic theory are a set of measure $1$ and therefore dense in $M$ (see e.g. 
\cite{FurstenbergRecurrence,DGS} for the ergodic theory part). 
Now define the sequence $g_k=g(T^k)$ of smooth maps $M \to \mathbb{R}$. 
Because $T$ is a diffeomorphism $M \to M$, the chain rule assures that the translated functions
$g_k$ are all Morse. Their indices are located on the points $T^k(m_j)$. By ergodicity 
already, the point measures $(1/n) \sum_{k=1}^n \delta_{m_j(T^k)}$ converge weakly 
to the constant function $1$ on $M$. 
The question is now whether there is an accumulation point of this sequence
$\mu_n = \frac{1}{n} \sum_{k=1}^n \delta_{g_k}$ on $\Omega$, where $\delta_{g_k}$ is 
the Dirac point measure located on $g_k$. 
This would only work if we had a weak-* compactness, but that requires a tightness
preventing the measure to escape.  The index expectation $K_n(v) = E_{\mu_n}[i(v)]$ 
would then converge to a $T$-invariant 
constant curvature measure $K=\chi(G)/Vol(M)$ which because of the normalization
${\rm Vol}(M)=\int_M 1 dV = 1$ satisfies $\int_M K \; dV=\chi(G)$. 

\paragraph{}
Which smooth functions $K$ on $M$ satisfying $\int_M K(x) \; dV(x) = \chi(G)$ 
can be realized as index expectation $K(x)=E_{\mu}[i(x)]$? 
Lets try to realize a function  $K$ which has no root and is close enough to a constant. 
because of compactness and connectedness of $M$, there exists a $\delta>0$ such
that $K(x) \geq \delta>0$ $-K(x) \geq \delta >0$ and all $x \in M$. If we define 
a new measure $dV$ which has higher weight somewhere, and do the above construction we have also a
higher curvature value there. Let $m=\chi(G)$ be the average curvature. 
Take the measure with density $dV(x) + (K(x)-m) dV$. Assume that this is positive
everywhere. Now pick the ergodic transformation from the Brin-Feldman-Katok theorem to get
an index expectation $K(x)$. 

\paragraph{}
We can also ask which manifolds allow for constant curvature $K_{\mu}$ with $\mu$ supported on 
a {\bf compact subset of $C^2$ functions} in $\Omega$.
We expect that not all manifolds allow for such measures and that the argument is similar
to the argument showing that there is no metric $g$ in general on a compact connected even
dimensional manifold for which the Gauss-Bonnet-Chern curvature is constant. 

\section{Constant curvature graphs}

\paragraph{} 
Before we look at the analogue curvature question in the discrete, let us
start with the question, which graphs have constant Euler curvature 
$$  K(v) = 1+\sum_{k=0} (-1)^k \frac{v_k(S(v))}{k+1}  \; , $$
where $v_k$ counts the number of $k$-dimensional simplices in the unit sphere 
$S(v)$ of $v$. (This curvature \cite{cherngaussbonnet} 
is the analogue of the Gauss-Bonnet-Chern measure in the continuum and appeared
already \cite{Levitt1992}). In one dimension, connected graphs with constant curvature are 
regular graphs like circular graphs $C_k$ with $k \geq 4$, the cube graph or 
dodecahedron graph, the tesseract graph, the complete bipartite graphs $K_{n,n}$ which 
includes $K_{1,1}=K_2$. For 2-graphs, connected examples are the icosahedron 
graph and the octahedron graph.

\question{
Can we characterize the set of connected finite simple graphs for which the Euler-Levitt 
curvature is constant? }

\paragraph{}
The class obviously contains all graphs $G$ for which all unit spheres $S(v)$ are isomorphic
to some fixed graph $H$. Even more generally, it contains all graphs for which
the $f$-vectors of $S(v)$ all agree. But we do not know for example, whether 
this is necessary nor whether this is sufficient to have constant Euler-Levitt curvature. 

\paragraph{}
The question of existence of constant curvature on a graph becomes richer if
curvature is formulated more broadly. We want the curvature function to be located 
on the vertex set $V$ and that it adds up to the Euler characteristic $\chi(G) = \sum_x \omega(x)$.
We want it to be local in the sense that it only depends on the unit sphere of $x$. We also want
it to be intrinsic in the sense that it does not depend on any auxiliary space like an embedding in 
some Euclidean space. 

\paragraph{}
Such curvatures can be obtained by distributing the values
$\omega(x) = (-1)^{{\rm dim}(x)}$ from a simplex $x$ to the vertices in $x$. In other words, 
we make each simplex $x$ a probability space and randomly distribute the ``energy value"
$\omega(x)$ to the zero-dimensional atoms of the simplex. This produces a curvature $K(v)$
on vertices which adds up to Euler characteristic. We probably got to this simple picture
of curvature in \cite{josellisknill} and not yet in \cite{cherngaussbonnet,poincarehopf}. 

\paragraph{}
This set-up is simple and assures that 
curvature remains local and unifies the continuum and discrete. 
In the continuum, it leads to the Poincar\'e-Hopf theorem if probabilities and so curvature
is integer-valued, meaning they are divisors. Then there is the case, where the probability
measures have a uniform distribution \cite{indexexpectation,colorcurvature}. 
In this case, we get the curvature we are used to in the continuum like the Gauss-Bonnet-Chern
integrand for Riemannian manifolds. 

\paragraph{}
Let $G=(V,E)$ be a finite simple graph. 
It turns out that if $G$ is $1$-dimensional, the Euler 
characteristic alone determines how big the solution space of probability measures is. 
Let $p=\{ p_x \; | \;  x \in G \}$ be the set of probability measures on 
the set $G_1$ of complete sub-graphs of a graph, defining the curvature $K$. Here
is a first result. The proof is given later. 

\begin{thm}
For a triangle-free connected graph, the set of measures $p$ producing constant 
curvature is a convex set of dimension $1-\chi(G)$. 
In particular, there is a unique measure $p$ for trees. There are examples of
graphs with triangles for which no measure $p$ produces constant curvature.
\end{thm} 

\paragraph{}
Integral geometry produces lots of opportunities. One important point is deformation. 
Both in the continuum as in the discrete, in order to deform space, we can either deform
the exterior derivative which changes distances via Connes metric or then we can 
deform probability measures defining quantities integral geometrically. 
In the first case, this can be done by differential equations, in the second case, where
we have a convex space of measures to play with, one can change distances or curvatures
with gradient flows. 

\section{Curvature}

\paragraph{}
Given a {\bf finite abstract simplicial complex} $G$ and a {\bf direction} $F:G \to V$, 
where $V = \bigcup_x x$ is the {\bf vertex set} of $G$, the index $i = F^*(\omega)$ 
is the push-forward of the signed measure $\omega(x)=(-1)^{{\rm dim}(x)}$ on $G$ which
as an integer-valued function can be seen as a divisor.  A special case is given by the
Whitney complex of a connected digraph graph $\Gamma$ without triangular cycles. In 
this case $F(x)$ is the largest element on $x$ in the partial order defined by the directions.
An even more special case is if $g$ is a locally injective function on $V$ and $F(x)$ is the 
vertex in $x$ where $F$ is maximal. We get then the Poincar\'e-Hopf index  which corresponds
to Poincar\'e-Hopf indices in the continuum.  See \cite{PoincareHopfVectorFields,MorePoincareHopf}
for more on this.

\paragraph{}
Averaging such indices over a probability space of directions $F$ on a graph 
produces a curvature $K(v)$ on $V$. This set-up is simple but it becomes on differentiable
manifolds the classical Poincar\'e-Hopf theorem for vector fields or the Gauss-Bonnet-Chern 
theorem. To see curvature as ``index expectation" is an integral geometric point of view.
The set-up allows to bridge the continuum and the discrete because the definitions of curvature
are then the same. For a differentiable manifold we can chose a probability space of Morse functions
$g$ for example and declare the expectation of $i_g(v)$ to be the curvature $K$ of $M$. If the probability
space is nice then $K$ is a smooth function as in differential geometry. If the probability space 
is the space of all linear functions on an ambient Euclidean space of a Nash embedding and the Haar 
induced measure is chosen on the linear functions, then $K$ is the Euler measure appearing in the 
Gauss-Bonnet-Chern theorem. 

\paragraph{}
Getting back to combinatorics, if $F$ is a Markov process, meaning that a probability vector 
$p_x$ is given on each simplex, then the energy $\omega(x) = (-1)^{\rm dim}(x)$ is distributed 
randomly from the simplices to the vertices, leading to a {\bf curvature}
$K(v) = \sum_{x, v \in x} \omega(x) p_x$. This can be abbreviated as 
$$   K=A \omega $$ 
for the matrix $A(v,x) = p_x(v)$ which is stochastic in the
sense that all column vectors of $A$ are probability vectors. Gauss-Bonnet is
$\sum_{v \in V} K(v)=\chi(G)$, where $\chi(G)=\sum_{x \in G} \omega(x)$. 
The columns of $A$ are probability vectors with the property that $p_x(v)=0$ if $v$
is not in $x$.  

\paragraph{}
If all probability distribution vectors $p_x$ are constant vectors and $G$ is the Whitney 
complex of a graph, then we get the Levitt curvature 
$$  K(v)=1+\sum_{k=0} (-1)^k v_k(S(v))/(k+1)  \; , $$ 
where $v_k(A)$ is the number of $k$-dimensional simplices in $A$ and $S(v)$ 
is the unit sphere of $v$. For $2$-graphs, graphs in which every
unit sphere $S(x)$ is a circular graph with $4$ or more vertices, this gives the curvature
$K(v) = 1-{\rm deg}(v)/6$ which has been known since at least a century.
Historically, it appeared already in \cite{Eberhard1891} and was then considered by 
Heesch \cite{Heesch}. 

\paragraph{}
One can now look at {\bf positive Euler curvature graphs}, which are graphs in 
which this particular curvature $K(v)$ is positive. 
An other point of view is to take a measure $\mu$ on the space of locally injective functions
$g$ and define graphs for which there exists $\mu$ with $K_{\mu}(v)>0$ everywhere. 
We still have to explore for which choices of probability measures $p_x$ one can get
a probability measure $\mu$ on the space $\Omega$ of locally injective functions (colorings), 
such that the induced measure on $x$ is $p_x$. There are obvious cases of choices of
probabilies $p_x$ such that for $y \subset x$ the probability $p_y$ is not 
compatible with $p_x$ preventing a realization as a measure on functions. In the one dimensional
case, where such compatibilities are not present, we can realize any $\{ p_x \}_{x \in G}$ choice
with a measure $\mu$ on $\Omega$. 

\section{Constant curvature}

\paragraph{}
An even more general case if to look at connected simplicial complexes $G$ and curvatures $K_P$ 
defined by having each set $x \in X$ equipped with a probability measure $p_x$. This triggers 
interest in graphs which have positive Euler characteristic but which do not allow for a
positive curvature. Here is a first basic question. If not stated otherwise, the 
simplicial complex of a graph is the Whitney complex. The measure $\mu$ represents a probability 
measure on ``discrete vector fields" which here is implemented through a family $p_x$ of probability
measures on the simplices of the graph. If the probability distributions come from 
a probability measure $\mu$ on locally injective functions on the graph, we speak of a
index expectation curvature. In the one-dimensional case mostly covered here, the two things
are the same. Every family of probability distributions $p_x$ on the edge sets can be realized
through a probability distribution on locally injective functions. But in general the question
can be different and is unexplored. 

\question{
Which connected graphs allow a constant $\mu$ curvature?
For which graphs is there a constant index expectation curvature. 
}

\paragraph{}
If $K$ is a constant curvature on a graph $(V,E)$, 
then by Gauss-Bonnet, it must have the value $\chi(G)/|V|$.
Complexes defined as Whitney complexes of small graphs like complete graphs, 
cyclic graphs, star graphs, wheel graphs, Platonic solid graphs allow for constant
curvature and $P$ is unique. But there are examples where constant curvature is not possible: 
take two graphs, where one has $\chi(G)=2$ and the other $\chi(H)=-2$
and connect them by an edge. This produces Euler characteristic $0$. But the
Euler characteristic of the positive curvature part remains positive after joining.
This is the idea behind the situation given in Figure~\ref{fishgraph}. The argument works
also for manifolds, but only for Gauss-Bonnet-Chern curvature. 
We can always realize constant $\mu$ curvature for compact connected Riemannian manifolds. 

\paragraph{}
We expect that for most graphs, constant curvature is not possible. 
For every vertex $v$, there is a bound which curvature can take on $v$. An upper bound is 
the number of positive dimensional simplices containing $v$, a lower bound 
is minus the number of negative dimensional simplices containing $v$. Let $M$ be the 
minimal bound and let $G$ be such that the average curvature value $\chi(G)/|V|$ is 
larger than $M$. Large curvature ratios are frequent as one can see by looking at
Euler characteristic averages on Erdoes-Renyi spaces of graphs or by taking joins of graphs
for which $\chi(G+H) = -\chi(G) \chi(H)+\chi(G) +\chi(H)$ and the number of vertices add
$|V(G+H)| = |V(G)| + |V(H)|$. 
One would have to distribute the values of the simplex curvatures to a larger neighborhood
in order to get a similar result as in the continuum and get constant curvature on any graph. 
 
\paragraph{}
Deciding whether a graph allows positive curvature is an inverse problem for
Markov processes. Given all the measures $p_x$, the
curvature $K$ is an equilibrium measure. The measures $p_x$ can be seen as column vectors in 
a stochastic $m \times n$ matrix, where $m=|V|$ is the number of vertices and $|G|$ is the
number of simplices in $G$. The probability vectors are the
columns of $A$ and we have $A(x,v)=0$ if $v$ is not in $x$.
We want to seen whether it is possible that $K=A \omega$ is constant. 

\begin{figure}[!htpb]
\scalebox{0.6}{\includegraphics{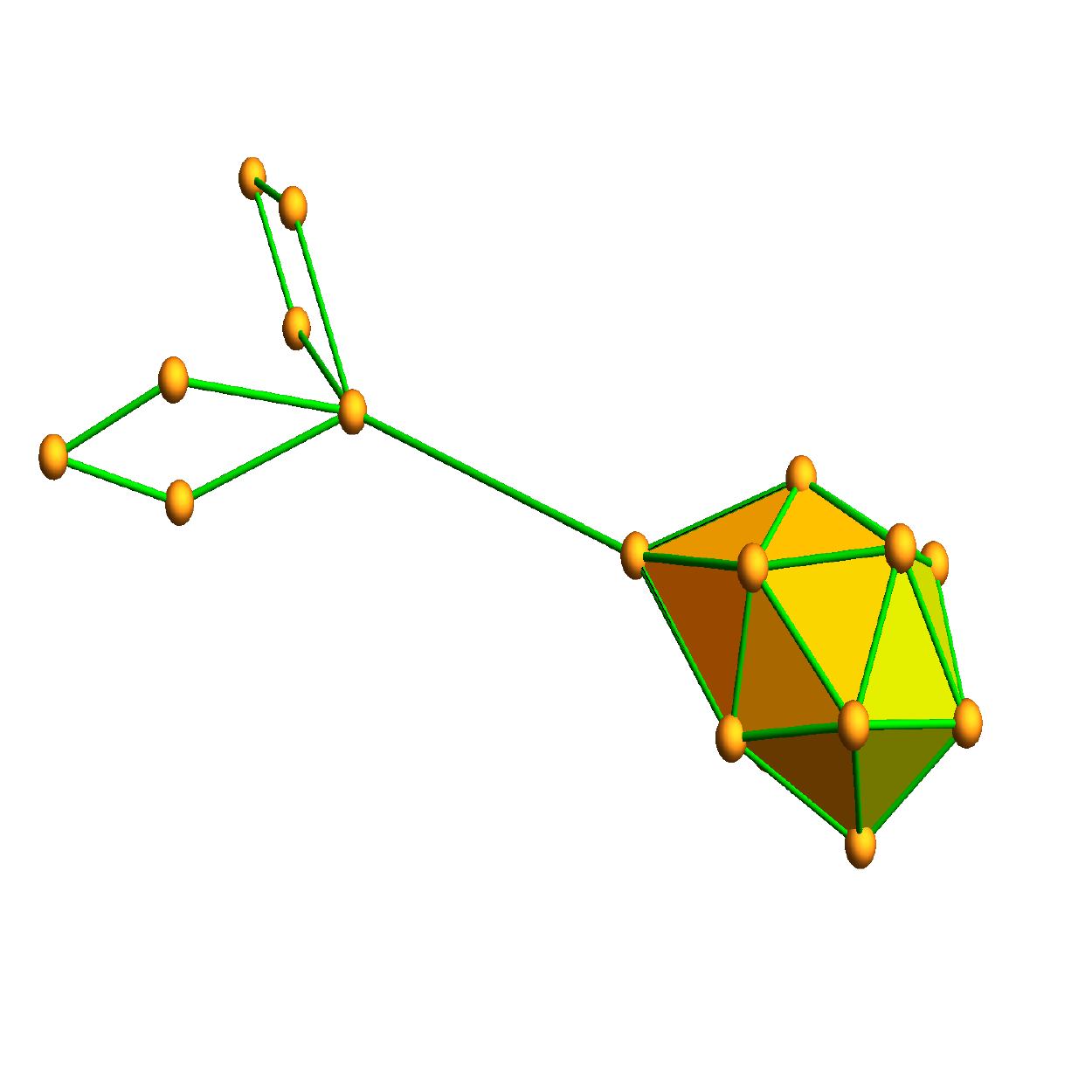}}
\label{fishgraph}
\caption{
This fish graph has Euler characteristic $0$ and does not allow for 
constant curvature. The curvature would have to be zero everywhere
but the separated fish body has Euler characteristic $2$ while the 
separated tail has Euler characteristic $-1$. Joining the two parts
by an edge lowers the Euler characteristic to $0$. The edge $e$ energy
$\omega(e)=-1$ can only be transported to the body or the tail. 
The body has now minimal Euler characteristic $1$.
}
\end{figure}

\section{Energized complexes}

\paragraph{}
The question can be generalized to energized complexes \cite{EnergizedSimplicialComplexes}. 
Let $G$ be a finite abstract simplicial
complex, $h: G \to \mathbb{C}$ an energy function and total energy 
$H(G)=\sum_x h(x)$ extending to $H(A) = \sum_{x \in A} h(x)$ on all sub complexes of $G$
Let $P$ assign to every simplex a probability distribution
so that $G$ is a collection of finite probability distributions. Define
as before the curvature $K=P H$. This means that the energy $h(x)$ is 
distributed randomly to vertices of $x$ according to the probability measure 
$p_x$ on $x$. Let us call $(G,H,P)$ an {\bf energized weighted simplicial complex}.

\paragraph{}
We can now think of a constant curvature as a type of {\bf equilibrium}.
The topological case $H(x)=\omega(x)$ is just a special case. 
Allowing the energy $H$ to be real or complex valued allows to think of 
$H$ as a {\bf wave amplitude} or wave as in quantum mechanics. The question
is now whether there is a way to guide the energy on every simplex so 
that the energy is the same on each vertex. 

\question{ Which energized weighted simplicial complexes $(G,H)$ allow for 
constant curvature?
}

\paragraph{}
Even for trees, the energies can not be too far away. Lets look at the 
case where $G=K_2$ is the one dimensional simplex, a very simple tree 
with $2$ vertices $\{1,2\}$ and $1$ edge $\{(1,2)\}$. Let $H(1)=a, H(2)=b$
and $H((1,2))=c$. Now, the total energy is $a+b+c$ and the constant curvature
would have to be $(a+b+c)/3$. We need now a $p \in [0,1]$ 
such that $a+p c= (a+b+c)/3$ and $b+(1-p) c = (a+b+c)/3$. 

\paragraph{}
One can make the problem more intricate by linking $H$ with $P$. 
One can for example define $H(x)$ to be the entropy 
$S(x) = -\sum_{v \in x} p_x(v) \log(p_x(v))$ of $p_x$. A 
variant is to take $H(x) = S(x) \omega(x)$. 
Now, the question is whether there exists a probability distribution
$p_x$ on each simplex such that its entropy is distributed evenly. 
By symmetry this happens for example if $G$ is the complete complex. 

\paragraph{}
There is an other variant in which we ask the curvature to stay quantized.
The Riemann-Roch theorem for graphs is related to a chip firing game in which
divisor values can be distributed to neighboring vertices. The analogue of a
divisor is a general integer valued function $H: G \to \mathbb{Z}$.
Given such a function $H$, we can ask to distribute $H$ to the vertices and
ask to minimize the variance. One could also ask to keep curvatures integers
leading to analogues of Poincar\'e-Hopf indices.

\section{Classical questions}

\paragraph{}
In the continuum, the analogue question is which
even dimensional manifold admit constant Euler curvature entering
Gauss-Bonnet-Chern $\int_M K \; dV=\chi(M)$. Related to the Hopf
conjectures is already whether there is a metric on $S^2 \times S^2$
with constant $K$:

\question{
For which manifolds is there a metric such that
the Gauss-Bonnet-Chern integrand for $(M,g)$ is constant? 
}

\paragraph{}
There are many questions in
differential geometry which asks for ``which compact Riemannian manifolds admit constant
curvature of some kind. Constant Ricci curvature gives Einstein manifolds.
Constant mean curvature surfaces produce minimal surfaces.
One can therefor ask the question for {\bf Euler curvature} which appears in the 
Gauss-Bonnet-Chern theorem $\int_M K(x) \; dV =\chi(M)$ for compact even dimensional
manifolds. 

\paragraph{}
Which compact 2-dimensional Riemannian manifolds have a constant Euler curvature? 
By uniformization and taking universal covers in the non-orientable case,
every connected two-dimensional compact manifold allows for a constant
Euler curvature metric.

\paragraph{} 
Here is a simple observation: 

\begin{propo}
There are compact $4$-dimensional Riemannian manifolds $M$ which 
do not admit a constant Euler curvature (Gauss-Bonnet-Chern integrand). 
\end{propo}

\begin{proof}
Take two compact connected 4-manifolds $M_1,M_2$, where $M_1$ has
Euler characteristic larger or equal than $4$ and $M_2$ has Euler characteristic smaller
or equal than $-4$. Now make a connected sum $M_1 +_B M_2$ along a 4-ball $B$ obtained by 
removing 4-balls $B_i$ from $M_i$ and gluing together along the boundary $3$-sphere. 
When combining, we lose the Euler characteristic of the two balls and have therefore $\chi(M)=\chi(M_1) + \chi(M_2)-2=-2$. 
By locality, we have to keep the Euler curvatures on $M_1$ and $M_2$. We can complement the 
connecting tube $N_0$ with mostly zero Euler curvature. Now $M=N_1 \cup N_2 \cup N_0$, where $N_i=M_i \setminus B$ 
and $N_0$ intersect in spheres. Now assume we can equip the connected sum $M$ with a constant 
curvature $K$. Because $\chi(M)=-2$ it is negative, the curvature $K$ would have to be negative. 
The manifold $N_1$ has now negative curvature in the interior and must by Gauss-Bonnet-Chern
have at least normalized curvature $4$ at the boundary. The boundary however bounds a 
4-ball $N_0$ of Euler characteristic $2$ which can be made arbitrary small so that the 
normalized boundary curvature can be made arbitrary close to $2$. 
That is incompatible with having to be $4$. 
\end{proof}

\paragraph{}
The argument can be done in any even dimension larger than $2$. 
This argument obviously does not work in dimension $2$ because connected $2$-manifolds
have Euler characteristic $\leq 2$. For example, when taking the connected sum of a genus $g_1$
surface with a genus $g_2$ surface, we get a surface with genus $(g_1 + g_2)$. It is only in 
dimension $4$ or higher that we can realize manifolds of Euler characteristic $4$. An example
is $S^2 \times S^2$. An concrete example of a $4$-manifold with Euler characteristic $-4$
is the Cartesian product $N \times N$ of two $2$-surfaces of genus $3$. 

\paragraph{}
Inverse questions about curvature are difficult in general. This is illustrated
by one of the celebrated open Hopf conjecture which asks whether there is a metric on $S^2 \times S^2$
for which the sectional curvature is positive. The question whether
there exists a metric with positive Euler curvature is not settled but we are also not aware whether
it is known whether there exists a metric on $S^2 \times S^2$ for which the Euler curvature
is constant. 

\section{Existence and uniqueness}

\paragraph{}
Let us look now at the discrete $1$-dimensional case, where we have a finite simple graph $(V,E)$.
Let $G$ be the simplicial Whitney complex which is here just the union of the vertex set $V$ and edge
set $E$. The function $\omega(x) = (-1)^{{\rm dim}(x)}$ on $G$ has the total sum $\sum_x \omega(x) = \chi(G)$. 
In order to realize constant curvature, we need to find a stochastic $|V| \times (|V| + |E|)$ matrix 
$A$ such that 
$$  A \omega = \chi(G)/|V| \;  $$
and such that $A_{vx}=0$ if $v$ is not a subset of $x$. 
This is a system of $m$ equations which together with the probability assumption $\sum_{v \in x} p_x(v)=1$
produces $m+n$ equations. There are $\sum_x {\rm dim}(x)+1$ unknowns $A_{ij}$. So, there is a 
large dimensional space of solutions but the question is whether we can get solutions 
satisfying $A_{ij} \geq 0$.  We see: 

\begin{propo}
The existence problem of constant $\mu$ curvature on a finite simple 
graph is a linear programming problem. 
\end{propo}

\paragraph{}
Example: 
Let $G=\{ (1),(2),(3),(12),(23)\}$ be the smallest line graph. There are two parameters $p,q$ and
we have to solve  $A \left[ \begin{array}{c} 1 \\ 1 \\ 1\\ -1 \\ -1 \end{array} \right] 
                  = \left[ \begin{array}{c} 1/3 \\ 1/3 \\ 1/3   \end{array} \right]$,
for the $3 \times 5$ matrix:
$$ A = \left[  \begin{array}{cccccc} 
              1 & 0 & 0 & p   & 0 \\
              0 & 1 & 0 & 1-p & q \\
              0 & 0 & 1 & 0   & 1-q \end{array} \right]  \; . $$
This leads to the three equation s $1-p=1/3, 1-(1-p)-q = 1/3, 1-(1-q)=1/3$
which gives $p=2/3,q=1/3$. 

\paragraph{}
Assume we find $P$ such that $P \omega = K$ is constant. Is this unique? 
The set of solutions $P$ is a convex set. The inverse problem for the
stochastic matrix $P$ with constraint that $p_x(v)=0$ if $v$ is not in $x$
can have infinitely many solutions or have a unique solution. 

\begin{thm}
If $G$ is a one-dimensional simplicial complex which is a tree, then 
there is a unique probability space in each edge for which curvature is constant. 
\end{thm}
\begin{proof}
Proof by induction. There is nothing to prove if the tree is a seed $K_1$, meaning
that there is no edge. The proof goes by induction with respect to the number of 
edges. Assume the statement has been proven for graphs with $n$ edges,
take a graph with $n+1$ edges, pick leaf, an end point $x$ of the tree which has only one neighbor. Since
the curvature has to be $1/|V|$ we know $p_x$ on the edge $e$ leading to $x$. This
defines the probability space on this edge. Now take the equilibrium measure which is 
given on the rest by induction. 
\end{proof} 

\paragraph{}
We can also see it from linear algebra: we have a free variable for every edge given $|E|$ variables.
and $|V|$ equations where one of them is automatically satisfied by the
Gauss-Bonnet formula. So, we have the same number of variables than equations. 
We can not have $|E| < |V|-1$ for one dimensional connected complexes
as this would imply $\chi(G)>1$ and we know that $\chi(G)=b_0-b_1 = 1-b_1 \leq 1$. 
Indeed, the Betti number $b_1$ matters:

\begin{thm}
For a one-dimensional complex, the solution space of probability measures $P$
satisfying the constant curvature equation $P \omega = K$ is $b_1$-dimensional. 
\end{thm}
\begin{proof}
Use induction with respect to the number $b_1$ of loops. It is easier for induction 
to prove a slightly stronger statement. Lets call a vertex $v$ in a graph with 
vertex degree $1$ a ``leave". Leaves are vertices with energy $\omega(v)=1$. We 
can more generally assign to a leaf an energy $H(v) \in [0,1]$. With $H(x)=\omega(x)$
for the other simplices in the graph, we have the total energy $H(G)=\sum_x H(x)$. 
It is now smaller than $\chi(G)$. Still, we can distribute the energies of the edges
around to get constant curvature. We prove this statement by induction.
In the case of $b_1=0$ loops we have a tree for which we have uniqueness in producing
the curvature. Also this case can be proven by induction with respect to the number of vertices: 
just cut one leave and the corresponding stem edge to that leave to get the statement for one vertex less.
For a tree we can even replace the energy $\omega(v)=1$
of a finite number $k$ of ``leaf vertices" (a vertex with vertex degree 1), with $\omega(v)=0$ and lower
the total energy from $\chi(G)=1$ to $1-k$. Now, we can realize the constant curvature $(1-k)/n$
everywhere.  By cutting a loop by removing an edge $e$ we obtain two more leaves 
and moving $\omega(e)=-1$ to one of the two vertices, we get a graph with one loop less for which
the energy at one leaf is $0$. 
For any generator $\gamma$ of the fundamental group we have a one-parameter choice to 
shift around probabilities around $\gamma$. This shows that we have a $b_1$ dimensional
space of solutions. Alternatively, we can make $b_1$ cuts to get a tree and for each
cut, removing an edge $e=(a,b)$, we have a choice how to distribute the energy 
$\omega(e)=-1$ to the two boundary points. 
\end{proof}

\paragraph{}
Here are some examples. \\

1) In the case of a star graph $S_3 = \{ (1),(2),(3),(4),(12),(13),(14) \}$ 
which is an example of a tree, we have a $4 \times 7$ matrix 
$$ P = \left[  \begin{array}{ccccccc} 
              1 & 0 & 0 & 0   & p   &  q  & r   \\
              0 & 1 & 0 & 0   & 1-p &  0  & 0   \\
              0 & 0 & 1 & 0   & 0   & 1-q & 0   \\ 
              0 & 0 & 0 & 1   & 0   &  0  & 1-r \\ 
  \end{array} \right]  \; . $$
The constant curvature equation
$P [1,1,1,1,-1,-1,-1]= [1,1,1,1]/4$ has only one solution $P$
which is given by $p=q=r=1/4$. 

\paragraph{}
2) In the case of a circular graph
$C_n$, we have an entire interval of solutions. Just pick the same probability
space in each case. For example, for the cyclic complex 
$C_3 = \{ (1),(2),(3),(12),(23),(31) \}$, we have 
$$ P = \left[  \begin{array}{cccccc} 
              1 & 0 & 0 & p   & 0   & 1-r \\
              0 & 1 & 0 & 1-p & q   &  0  \\
              0 & 0 & 1 & 0   & 1-q &  r  \\ \end{array} \right]  \; . $$

\paragraph{}
3) Let us now look at the case of the triangle $K_3$ which is obtained 
from $C_3$ by adding a $2$-dimensional cell 
$K_3 = \{ (1),(2),(3),(12),(23),(31),(123)\}$. We have 
$$ P = \left[  \begin{array}{ccccccc} 
              1 & 0 & 0 & p   & 0   & 1-r & s \\
              0 & 1 & 0 & 1-p & q   &  0  & t \\
              0 & 0 & 1 & 0   & 1-q &  r  & 1-s-t \\ \end{array} \right]  \; . $$
In this case, we can chose $p,q,r$ and $t,s$ is determined. Since there are three
free variables $p,q,r$, there is also here no uniqueness. In general, for higher dimensional
spaces it is easier to distribute out curvature so that it becomes constant. 

\bibliographystyle{plain}

\end{document}